\documentclass[12pt,reqno]{amsart}
\usepackage{amsfonts,amssymb,amsbsy,amsmath}

\advance \topmargin by -\headheight
\advance \topmargin by -\headsep
     
\evensidemargin \oddsidemargin
\newtheorem{theorem}{Theorem}[section]
\newtheorem{lemma}[theorem]{Lemma}

\theoremstyle{definition}

\theoremstyle{remark}

\numberwithin{equation}{section}

\newcommand{\mmod}[1]{\,\,(\text{mod}\,\,#1)}

\def\bfx{{\mathbf x}}
\def\bfy{{\mathbf y}}

\def\calM{{\mathcal M}}

\def\dbN{{\mathbb N}}

\def\dbZ{{\mathbb Z}}\def\dbQ{{\mathbb Q}}

\def\alp{{\alpha}}

\def\tet{{\theta}}

\def\le{\leqslant} \def\ge{\geqslant}

\begin{document}
\title[Artin's Conjecture]{Artin's Conjecture and systems\\ of diagonal equations}
\author[T. D. Wooley]{Trevor D. Wooley}
\address{School of Mathematics, University of Bristol, University Walk, Clifton, Bristol BS8 1TW, United Kingdom}
\email{matdw@bristol.ac.uk}
\date{}
\begin{abstract} We show that Artin's conjecture concerning $p$-adic solubility of Diophantine equations fails for infinitely many systems of $r$ homogeneous diagonal equations whenever $r\ge 2$.\end{abstract}
\maketitle

\section{Introduction} A modern formulation of Artin's Conjecture (AC) asserts that in the $p$-adic field $\dbQ_p$, any
 collection of forms $F_1,\ldots ,F_r\in \dbQ_p[x_1,\ldots ,x_s]$ of respective degrees $d_1,\ldots ,d_r$ must possess a 
common non-trivial zero provided only that $s>d_1^2+\ldots +d_r^2$. This is in fact a variant of a special case of the
 conjecture made by E. Artin (see \cite[Preface page x]{Art1965} and \cite{Lan1952} for more on the subtext to this
 statement). The history of this conjecture is remarkably rich. For the purpose at hand, it suffices to note that AC is known to
 hold whenever $p$ is sufficiently large in terms of the degrees of the forms at hand (see \cite{AK1965,Bro1978}), yet fails
 spectacularly for each prime $p$ (see \cite{AK1981,Bro1984,LM1983,Ter1966}). The conjecture is known to hold, however, for
 systems consisting of a single quadratic or cubic form, a pair of quadratic forms, and also in the case of a single diagonal form
 of any degree (see \cite{DL1963,Dem1956,Has1924,Lew1952}). The latter conclusion has prompted speculation
 that AC might be salvaged for systems of diagonal forms. Workers intimately familiar with the argument underlying
 \cite{LM1983} have long been aware that AC cannot hold for certain systems of diagonal forms in which sufficiently many 
distinct degrees of suitable size occur, although this observation does not seem to be particularly well-documented in the
 literature. However, for larger degrees such counter-examples to AC necessarily contain very many forms. Our goal in this note 
is to show that counter-examples to AC exist in abundance for systems of two or more diagonal forms of differing degrees.

\begin{theorem}\label{theorem1.1} Let $p$ be a prime number. Then whenever $r\ge 2$, there are infinitely many $r$-tuples 
$(d_1,\ldots ,d_r)$ having the property that Artin's Conjecture fails over $\dbQ_p$ for a system of diagonal forms
 of respective degrees $d_1,\ldots ,d_r$.
\end{theorem}

There has been much work in the recent literature concerning systems of diagonal equations and their relation to AC (see
 \cite{BG1999,BG2002,Kna2001,Kna2007a, Kna2007b,Kna2009,Kna2012}). In particular, it is known that AC holds for pairs of
 diagonal forms of equal degree $k$, except possibly when $k$ takes the form $p^\tau(p-1)$ or $3\cdot 2^\tau$ ($p=2$), 
and that AC holds also for pairs of diagonal forms of distinct odd degrees. While our theorem rules out the validity of AC in 
general for systems of two or more diagonal forms, it remains conceivable that AC holds for systems of odd degree.\par

In order to establish our theorem, we adapt the strategy of Lewis and Montgomery \cite{LM1983} to obtain a strengthened 
$p$-adic interpolation lemma. Let $p$ be a prime number and $h$ a natural number, and write $\phi(p^h)=p^{h-1}(p-1)$. 
When $N\ge 1$, define $S_{h,m}(\bfx)$ by putting
$$S_{h,m}(x_1,\ldots ,x_N)=\sum_{i=1}^Nx_i^{\phi(p^h)m}.$$
Then by enhancing the argument of \cite[Lemma 2]{LM1983}, we are able to establish the following conclusion.

\begin{theorem}\label{theorem1.2}
Let $p$ be a prime number and $h\in \dbN$. When $p=2$, suppose in addition that $h\ge 3$. Let $M$ be a positive 
integer, and consider a set $\calM$ of $r$ integers in the interval $[M,2M)$. Write $W=[\phi(p^h)/(h+1)]$. Finally, 
suppose that there is a non-zero solution $\bfx\in \dbQ_p^{Ws}$ to the system of equations
\begin{equation}\label{1.1}
\sum_{l=0}^{W-1}p^{l(h+1)M}S_{h,m}(x_{1l},\ldots ,x_{sl})=0\quad (m\in \calM).
\end{equation}
Then one has $s\ge p^{rh}$.
\end{theorem}

Note that the number of variables in the system (\ref{1.1}) is $Ws$. Also, when $h\ge 5$, the integer $W$ occurring in the 
statement of Theorem \ref{theorem1.2} satisfies
$$W\ge (p^{h-1}-h-1)/(h+1)\ge p^{h-2}/(h+1).$$
The validity of AC for the system (\ref{1.1}) would imply that a non-zero solution $\bfx\in \dbQ_p^{Ws}$ exists whenever
$$Ws\ge 4M^3p^{2h}>\sum_{m\in \calM}(p^{h-1}(p-1)m)^2.$$
However, the conclusion of Theorem \ref{theorem1.2} implies that no solution exists when $s<p^{rh}$. Consequently, 
whenever
$$p^{rh}>\frac{4M^3p^{2h}}{p^{h-2}/(h+1)}=4(h+1)M^3p^{h+2},$$
one finds that AC fails for the system (\ref{1.1}), and such is the case whenever $r>1$ and $h$ is sufficiently large in terms of $M$. The conclusion of Theorem \ref{theorem1.1} follows at once.\par

Throughout this note, we write $[\tet]$ for the greatest integer not exceeding $\tet$. 

\section{The proof of Theorem \ref{theorem1.2}.}
We begin by describing a variant of a result on $p$-adic interpolation provided by Lewis and Montgomery (see 
\cite[Lemma 1]{LM1983}). The latter authors established the case $h=1$ of the following conclusion. Here, as usual, when 
$\alp=p^ka/b\in \dbQ$, with $(a,p)=(b,p)=1$ and $k\in \dbZ$, we write $\text{ord}\, \alp$ for $k$.

\begin{lemma}\label{lemma2.1}
Let $a$ be an integer and let $n_1,\ldots ,n_K$ be distinct integers with $n_k\equiv a\mmod{p^h}$ $(1\le k\le K)$, for some 
$h\in \dbN$. Let $f\in \dbZ[z]$, and suppose that
$$f(n_k)\equiv 0\mmod{p^M}\quad (1\le k\le K).$$
Then one has
$$\text{ord}\, f(a)\ge \min \{Kh, (K-1)h-L+M\},$$
where
$$L=\max_{1\le k\le K}\Bigl\{ \text{ord}\Bigl( \prod_{\substack{j=1\\ j\ne k}}^K(n_j-n_k)\Bigr) \Bigr\}.$$
\end{lemma}

\begin{proof} One may follow the argument of the proof of \cite[Lemma 1]{LM1983} without serious modification to 
accommodate arbitrary values of $h$ in place of the special case $h=1$.
\end{proof}

We apply this conclusion to deduce a generalisation of \cite[Lemma 2]{LM1983}. The argument of our proof is very similar to 
that of the latter, but differs in enough detail that a complete account seems warranted.

\begin{lemma}\label{lemma2.2} Suppose that $p$ is a prime number and $h\in \dbN$. When $p=2$, suppose in addition 
that $h\ge 3$. Let $M$ be a positive integer, and let $\calM$ be a set of $K$ integers in the interval $[M,2M)$. Finally, suppose 
that there are $N$ integers $x_1,\ldots ,x_N$, not all divisible by $p$, such that
\begin{equation}\label{2.1}
S_{h,m}(x_1,\ldots ,x_N)\equiv 0\mmod{p^{(h+1)M}}\quad (m\in \calM).
\end{equation}
Then one has $N\ge p^{Kh}$.
\end{lemma}

\begin{proof} Write $q=\phi(p^h)$ and suppose that $\calM=\{m_1,\ldots ,m_K\}$. We begin by considering the situation in 
which $p$ is an odd prime. Since $x_n$ plays no role in the congruences (\ref{2.1}) when $p|x_n$, we may suppose without loss that $(x_n,p)=1$ for $1\le n\le N$. Let $g$ be a primitive root modulo $p^2$, whence $g$ is primitive for all powers
 of $p$. For each index $n$, there exists an integer $a_n$ with $0\le a_n<\phi(p^{(h+1)M})$ having the property that 
$x_n\equiv g^{a_n}\mmod{p^{(h+1)M}}$. Put $f(z)=\sum_{n=1}^Nz^{a_n}$, and note that $f(1)=N$. By the primitivity of 
$g$, one sees that for $1\le k\le K$ the integers $n_k=g^{qm_k}$ are distinct and satisfy the congruence 
$n_k\equiv 1\mmod{p^h}$. Furthermore, by hypothesis, one finds that $f(g^{qm})\equiv 0\mmod{p^{(h+1)M}}$ for 
$m\in \calM$. We therefore deduce from Lemma \ref{lemma2.1} that
\begin{equation}\label{2.2}
\text{ord}\, N=\text{ord}\, f(1)\ge \min \{ Kh, (K-1)h-L+(h+1)M\},
\end{equation}
where
$$L=\max_{m\in \calM}\text{ord}\Bigl( \prod_{\substack{r\in \calM\\ r\ne m}}(g^{qr}-g^{qm})\Bigr) .$$

\par Next we observe that since $q=\phi(p^h)$, one has $\text{ord}(g^{qs}-1)=h+\text{ord}\, s$ for every natural number 
$s$, whence $\text{ord}(g^{qr}-g^{qm})=h+\text{ord}(r-m)$. But when $m\in \calM$, just as in the argument of the proof 
of \cite[Lemma 2]{LM1983}, a consideration of binomial coefficients reveals that
$$\text{ord}\Bigl( \prod_{\substack{r\in \calM\\ r\ne m}}(r-m)\Bigr) \le \text{ord}\left( (m-M)!(2M-m-1)!\right) \le \text{ord}\left( (M-1)!\right) .$$
Hence we obtain
$$L\le (K-1)h+\sum_{j=1}^\infty \left[ \frac{M-1}{p^j}\right] \le (K-1)h+\frac{M-1}{p-1}.$$
One therefore has
$$(K-1)h-L+(h+1)M\ge \left( h+1-1/(p-1)\right) M\ge hM\ge hK.$$
Thus we conclude from (\ref{2.2}) that $\text{ord}\, N\ge hK$, so that $p^{hK}|N$.\par

When $p=2$, we modify the above argument by using the fact that the reduced residues are generated by $-1$ and $5$ 
modulo any power of $2$. Thus, for each index $n$, there exists an integer $a_n$ with 
$0\le a_n<\frac{1}{2}\phi(2^{(h+1)M})$ having the property that $x_n^2\equiv 5^{a_n}\mmod{2^{(h+1)M}}$. Put $f(z)=\sum_{n=1}^Nz^{a_n}$ and note once more 
that $f(1)=N$. For $1\le k\le K$, the integers $n_k=5^{m_kq/2}$ are distinct and satisfy the congruence 
$n_k\equiv 1\mmod{2^h}$. Further, by hypothesis, for $1\le k\le K$ one has
$$f(5^{m_kq/2})=\sum_{n=1}^Nx_n^{qm_k}\equiv 0\mmod{2^{(h+1)M}}.$$
It therefore follows that $f(n_k)\equiv 0\mmod{2^{(h+1)M}}$ for $1\le k\le K$. As before, we find from Lemma \ref{lemma2.1} 
that the lower bound (\ref{2.2}) holds, though in the present circumstances in which $p=2$, one has $L\le (K-1)h+M-1$. 
Consequently, we deduce on this occasion that since $h\ge 3$, one has
$$(K-1)h-L+(h+1)M\ge hM\ge hK.$$
Thus we conclude that $2^{hK}|N$, thereby completing the proof of the lemma.
\end{proof}

We are now equipped to prove Theorem \ref{theorem1.2}. Suppose if possible that $s<p^{rh}$ and that the equations 
(\ref{1.1}) have a non-zero simultaneous solution $\bfx\in \dbQ_p^{Ws}$. By homogeneity, we may suppose that 
$\bfx\in \dbZ_p^{Ws}$, and further that $p\nmid x_{ju}$ for some indices $j$ and $u$ with $1\le j\le s$ and $0\le u<W$. 
For some index $l$ with $0\le l<W$, one has $p|x_{iu}$ for $u<l$ and $1\le i\le s$, and further 
$p\nmid x_{jl}$ for some index $j$ with $1\le j\le s$. For each $m\in \calM$, we thus have
$$S_{h,m}(x_{1u},\ldots ,x_{su})\equiv 0\mmod{p^{\phi(p^h)M}}\quad (u<l).$$
On the other hand, since
$$\phi(p^h)-(h+1)l\ge \phi(p^h)-(h+1)(W-1)\ge h+1,$$
we find from (\ref{1.1}) that there is a solution $\bfx_l\in \dbZ_p^s$ of the simultaneous congruences
$$S_{h,m}(x_{1l},\ldots ,x_{sl})\equiv 0\mmod{p^{(h+1)M}}\quad (m\in \calM)$$
in which $p\nmid x_{jl}$ for some index $j$ with $1\le j\le s$. In particular, there is a non-zero $s$-tuple $\bfy\in \dbZ^s$ 
with $\bfy\equiv \bfx_l\mmod{p^{(h+1)M}}$ such that
$$S_{h,m}(\bfy)\equiv 0\mmod{p^{(h+1)M}}\quad (m\in \calM).$$
Since $s<p^{rh}$, it follows from Lemma \ref{lemma2.2} that one must have $p|y_i$, and hence also $p|x_{il}$, for 
$1\le i\le s$. This contradicts our earlier assumption, and so we reach a contradiction. We are therefore forced to conclude that 
$s\ge p^{rh}$, and this completes the proof of Theorem \ref{theorem1.2}.

\bibliographystyle{amsbracket}
\providecommand{\bysame}{\leavevmode\hbox to3em{\hrulefill}\thinspace}

\end{document}